\documentclass[11pt,reqno]{amsart}
\usepackage{amsmath,amssymb,graphicx,amscd}
\usepackage{mathrsfs}
\usepackage{mathrsfs}
\usepackage{dsfont}
\usepackage{stmaryrd}
\usepackage{wrapfig}
\usepackage{pstricks}

\renewcommand{\epsilon}{\varepsilon}
\renewcommand{\phi}{\varphi}
\renewcommand{\o}{\mathrm{o}}
\newcommand{\CC}{{\mathbb C}}

\newcommand{\RR}{{\mathbb R}}
\newcommand{\QQ}{{\mathbb{Q}}}
\newcommand{\NN}{{\mathbb{N}}}
\newcommand{\ZZ}{{\mathbb{Z}}}

 \DeclareMathOperator{\E}{\mathbb{E}}      
 \DeclareMathOperator{\cov}{cov}             
\DeclareMathOperator{\Var}{\mathrm{Var}}      
\DeclareMathOperator{\Tr}{Tr}

\newcommand{\tra}[1]{\,{\vphantom{#1}}^{\mathrm{t}}\hspace{0mm}{#1}}

\newcommand{\1}{\mathds{1}}
 \newcommand{\dd}{{\mathrm{d}}}            
 \newcommand{\ii}{{\mathrm{i}}}
  \newcommand{\Id}{{\mathrm{Id}}}

\newcommand{\convlaw}{\overset{\mbox{\rm \scriptsize law}}{\longrightarrow}}

\renewcommand{\Re}{{\mathfrak{Re}}}
\renewcommand{\Im}{{\mathfrak{Im}}}

\newtheorem{thm}{Theorem}[section]
\newtheorem{cor}[thm]{Corollary}
\newtheorem{lem}[thm]{Lemma}
\newtheorem{prop}[thm]{Proposition}
\theoremstyle{definition}
\newtheorem{defn}[thm]{Definition}
\theoremstyle{remark}
\newtheorem*{rem}{Remark}

\numberwithin{equation}{section}

\date{February 9, 2009}

\begin{document}
\title[Mesoscopic fluctuations of the zeta zeros]
{Mesoscopic fluctuations of the zeta zeros}

\author{P. Bourgade}
 \address{Telecom ParisTech,
 46 rue Barrault, 75634 Paris Cedex 13.}
\email{bourgade@enst.fr}


\begin{abstract}
We prove a multidimensional extension of Selberg's central limit theorem for $\log\zeta$, in which
non-trivial correlations appear.
In particular, this answers a question by Coram and Diaconis about the mesoscopic fluctuations of the
zeros of the Riemann zeta function.

Similar results are given in the context of random matrices from the unitary group.
This shows the correspondence $n \leftrightarrow \log t$ not only between the dimension of the matrix and the height
on the critical line, but also, in a local scale, for small deviations from the critical axis or the unit circle.

\end{abstract}

\maketitle

\begin{rem}
All results below hold for L-functions from the Selberg class, for
concision we state them for $\zeta$.

In this paper we talk about correlations between random variables to express the idea of dependence, which is equivalent as all the
involved variables are Gaussian.

The Vinogradov symbol, $a_n\ll
b_n$, means $a_n=O(b_n)$, and $a_n\gg b_n$ means $b_n\ll a_n$. In
this paper, we implicitly assume that, for all $n$ and $t$,
$\epsilon_n\geq 0$, $\epsilon_t\geq 0$.
\end{rem}

\section{Introduction}

\subsection{Main result}

Selberg's central limit theorem states that, if $\omega$ is uniform on
$(0,1)$, then
\begin{equation}\label{eqn:SelbergCLT}
\frac{\log \zeta\left(\frac{1}{2}+\ii \omega
t\right)}{\sqrt{\log\log t}}\convlaw Y,
\end{equation}
as $t\to\infty$, $Y$ being a standard complex normal
variable (see paragraph \ref{subsection:Definitions} below for precise definitions of $\log\zeta$ and complex normal variables).
This result has been extended in two distinct directions, both relying on Selberg's original method.

First similar central limit theorems appear in  Tsang's thesis \cite{Tsang}
far away from the critical axis, and Joyner \cite{Joyner} generalized these results to a larger class of L-functions.
In particular,  (\ref{eqn:SelbergCLT}) holds also for $\log\zeta$ evaluated close to the critical axis
($1/2+\epsilon_t+\ii\omega t$) provided that $\epsilon_t\ll 1/\log t$; for $\epsilon_t\to 0$ and $\epsilon_t\gg 1/\log t$,
 Tsang proved that a change of normalization is necessary:
\begin{equation}\label{eqn:GeneralSelbergCLT}
\frac{\log \zeta\left(\frac{1}{2}+\epsilon_t+\ii \omega
t\right)}{\sqrt{-\log\epsilon_t}}\convlaw Y',
\end{equation}
with $\omega$ uniform on $(0,1)$ and $Y'$ a standard complex normal variable.

Second, a multidimensional extension of (\ref{eqn:SelbergCLT}) was given by
Hughes, Nikeghbali and Yor \cite{HNY}, in order to get a dynamic analogue of  Selberg's central limit theorem :
they showed that for any $0<\lambda_1<\dots<\lambda_\ell$
\begin{multline}\label{eqn:HNYCLT}
\frac{1}{\sqrt{\log\log t}}\left(\log \zeta\left(\frac{1}{2}+\ii \omega
e^{(\log t)^{\lambda_1}}\right),\dots,\log \zeta\left(\frac{1}{2}+\ii \omega
e^{(\log t)^{\lambda_\ell}}\right)\right)\\
\convlaw
\left(\lambda_1Y_1,\dots,
\lambda_\ell Y_\ell\right),
\end{multline}
all the $Y_k$'s being independent standard complex normal variables.
The evaluation points $\frac{1}{2}+\ii \omega
e^{(\log t)^{\lambda_k}}$ in the above formula are very distant from each other and a natural question is whether, for closer points,
a non-trivial correlation structure appears for the values of zeta.
Actually, the average values of $\log\zeta$ become correlated for small shifts, and the Gaussian kernel appearing
in the limit coincides with the one of Brownian motion off the diagonal.
More precisely, our main result is the following.

\begin{thm}\label{thm:jointNT}
Let $\omega$ be uniform on $(0,1)$, $\epsilon_t\to 0$, $\epsilon_t\gg 1/\log t$, and functions
$0\leq f^{(1)}_t< \dots< f^{(\ell)}_t<c<\infty$.
Suppose that for all $i\neq j$
\begin{equation}\label{eqn:conditionNT}
\frac{\log|f^{(j)}_t-f^{(i)}_t|}{\log \epsilon_t}\to c_{i,j}\in[0,\infty].
\end{equation}
Then the vector
\begin{equation}\label{eqn:resultNT}
\frac{1}{\sqrt{-\log\epsilon_t}}\left(\log \zeta\left(\frac{1}{2}+\epsilon_t+\ii f^{(1)}_t+\ii \omega
t\right),\dots,
\log \zeta\left(\frac{1}{2}+\epsilon_t+\ii f^{(\ell)}_t+\ii \omega
t\right)\right)
\end{equation}
converges in law to a complex Gaussian vector $(Y_1,\dots,Y_\ell)$ with mean 0 and covariance function
\begin{equation}\label{eqn:covariance}
\cov(Y_i,Y_j)=\left\{
\begin{array}{ccc}
1&if&i=j\\
1\wedge c_{i,j}&if&i\neq j
\end{array}
\right..
\end{equation}
Moreover, the above result remains true if $\epsilon_t\ll 1/\log t$,
replacing the normalization $-\log \epsilon_t$ with $\log\log t$ in (\ref{eqn:conditionNT}) and (\ref{eqn:resultNT}).
\end{thm}

The covariance structure (\ref{eqn:covariance}) of the limit
Gaussian vector actually depends only on the $\ell-1$ parameters
$c_{1,2},\dots,c_{\ell-1,\ell}$ because formula
(\ref{eqn:conditionNT}) implies, for all $i<k<j$,
$c_{i,j}=c_{i,k}\wedge c_{k,j}$. We will explicitly construct
Gaussian vectors with the correlation structure
(\ref{eqn:covariance}) in section
\ref{Section:SpatialBranchingProcesses}.

We now illustrate Theorem \ref{thm:jointNT}. Take $\ell=2$, $\epsilon_t\to 0$, $\epsilon_t\gg 1/\log t$.
Then for any $0\leq \delta\leq 1$ and
$\omega$ uniform on $(0,1)$, choosing $f^{(1)}_t=0$ and $f^{(2)}_t=\epsilon_t^\delta$,
$$
\frac{1}{\sqrt{-\frac{1}{2}\log\epsilon_t}}\left(\log \left|\zeta\left(\frac{1}{2}+\epsilon_t+\ii \omega
t\right)\right|,
\log \left|\zeta\left(\frac{1}{2}+\epsilon_t+\ii \omega
t+\ii \epsilon_t^\delta\right)\right|\right)
$$
converges in law to
\begin{equation}\label{eqn:GaussianLimit}
(\mathcal{N}_1,\delta\mathcal{N}_1+\sqrt{1-\delta^2}\mathcal{N}_2),
\end{equation}
where $\mathcal{N}_1$ and $\mathcal{N}_2$ are independent standard real normal variables.
A similar result holds if $\epsilon_t\ll 1/\log t$, in particular we have a central limit theorem on the critical axis $\epsilon_t=0$ :
$$
\frac{1}{\sqrt{\frac{1}{2}\log\log t}}\left(\log \left|\zeta\left(\frac{1}{2}+\ii \omega
t\right)\right|,
\log \left|\zeta\left(\frac{1}{2}+\ii \omega
t+\frac{\ii}{(\log t)^\delta}\right)\right|\right)
$$
also converges in law to (\ref{eqn:GaussianLimit}).
Note the change of normalization
according to $\epsilon_t$, i.e. the distance to the critical axis.
Finally, if all shifts $f^{(i)}_t$ are constant and distinct, $c_{i,j}=0$ for all $i$ and $j$, so
the distinct means of $\zeta$ converge in law to independent complex normal variables,
after normalization.

\begin{rem}
In this paper we are concerned with distinct shifts along the ordinates, in particular because it implies
the following Corollary \ref{cor:CountingNT} about counting the zeros of the zeta function.
The same method equally applies to distinct shifts along the abscissa, not enounced here for simplicity. For example,
the Gaussian variables $Y$ and $Y'$ in (\ref{eqn:SelbergCLT}) and (\ref{eqn:GeneralSelbergCLT}) have correlation
$1\wedge\sqrt{\delta}$ if $\epsilon_t=1/(\log t)^\delta$ with $\delta>0$.
\end{rem}

Theorem \ref{thm:jointNT} can be understood in terms of Gaussian processes :
it has the following immediate consequence, enounced for $\epsilon_t=0$ for simplicity.

\begin{cor}\label{cor:StrangeProcess}
Let $\omega$ be uniform on $(0,1)$. Consider the random function
$$\left(\frac{1}{\sqrt{\log\log t}}\log\left|\zeta\left(\frac{1}{2}+\ii\omega t+\frac{\ii}{(\log t)^\delta}\right)\right|,0\leq \delta\leq 1\right)$$
Then its finite dimensional distribution converge, as $t\to\infty$, to those of a centered Gaussian process with kernel
$\Gamma_{\gamma,\delta}=\gamma\wedge \delta$ if $\gamma\neq \delta$, 1 if $\gamma=\delta$.
\end{cor}

There is an effective construction of a centered Gaussian process
$(X_\delta,0\leq \delta\leq 1)$ with covariance function $\Gamma_{\gamma,\delta}$ : let
$(B_\delta,0\leq \delta\leq 1)$ be a standard Brownian motion and independently let $(D_\delta,0\leq \delta\leq 1)$ be a totally disordered process,
meaning that all its coordinates are independent centered Gaussians with variance $\E(D_\delta^2)=\delta$. Then
$$X_\delta=B_\delta+D_{1-\delta}$$
defines a Gaussian process with the desired covariance function. Note that there is no measurable version of this process :
if there were, then  $(D_\delta,0\leq\delta\leq 1)$ would have a measurable version which is absurd because, by Fubini's Theorem,
for all $0\leq a<b\leq 1$
$\E\left(\left(\int_{a}^b D_\delta\dd\delta\right)^2\right)=0$, so $\int_{a}^b D_\delta\dd\delta=0$ a.s. and $D_\delta=0$ a.s.
giving the contradiction.\\

\subsection{Counting the zeros}

 Theorem \ref{thm:jointNT} also has a strange consequence for the counting of zeros of $\zeta$ on intervals
in the critical strip. Write $N(t)$ for
the number of non-trivial zeros $z$ of $\zeta$ with $0< \Im z \leq t$, counted with their multiplicity. Then
(see e.g. Theorem 9.3 in Titchmarsh \cite{Tit})
\begin{equation}\label{eqn:NdeT}
N(t)=\frac{t}{2\pi}\log\frac{t}{2\pi e}+\frac{1}{\pi}\Im\log\zeta\left(1/2+\ii t\right)+\frac{7}{8}+O\left(\frac{1}{t}\right)
\end{equation}
with $\Im\log\zeta\left(1/2+\ii t\right)=O(\log t)$. For $t_1<t_2$ we will write
$$\Delta(t_1,t_2)=\left(N(t_2)-N(t_1)\right)-\left(\frac{t_2}{2\pi}\log\frac{t_2}{2\pi e}-\frac{t_1}{2\pi}\log\frac{t_1}{2\pi e}\right),$$
which represents the fluctuations of the number of zeros $z$ ($t_1< \Im z \leq t_2$) minus its {\it expectation}.
A direct consequence of Theorem \ref{thm:jointNT}, choosing $\ell=2$, $f^{(1)}(t)= 0$ and
$f^{(2)}(t)=\frac{1}{(\log t)^\delta}$ ($0\leq \delta\leq 1$), is the following central limit theorem obtained by Fujii \cite{Fujii}:
$$
\frac{\Delta\left(\omega t,\omega t+\frac{1}{(\log t)^\delta}\right)}{\frac{1}{\pi}\sqrt{\log\log t}}\convlaw \sqrt{1-\delta}\mathcal{N}
$$
as $t\to\infty$, where $\omega$ is uniform on $(0,1)$ and
$\mathcal{N}$ is a standard real normal variable. A more general
result actually holds, being a direct consequence of Theorem
\ref{thm:jointNT} and (\ref{eqn:NdeT}). This confirms numerical experiments by
Coram and Diaconis \cite{CD}, who after making extensive tests (based on
data by Odlyzko) suggested that the correlation structure
(\ref{eqn:strangeCorrelation}) below should appear when counting the
zeros of $\zeta$. Following \cite{CD} the phenomenon presented below can be
seen as the {\it mesoscopic} repulsion of the zeta zeros, different
from the Montgomery-Odlyzko law, describing the repulsion at a
microscopic scale.

\begin{cor}\label{cor:CountingNT}
Let $(K_t)$ be such that, for some $\epsilon>0$ and all $t$, $K_t>\epsilon$. Suppose
$\log K_t/\log\log t\to \delta\in[0,1)$ as $t\to\infty$.
Then the finite dimensional distributions of the process
$$\frac{\Delta\left(\omega t+\alpha/K_t,\omega t+\beta/K_t\right)}{\frac{1}{\pi}\sqrt{(1-\delta)\log\log t}},\ 0\leq\alpha< \beta<\infty$$
converge to those of a centered Gaussian process $(\tilde\Delta(\alpha,\beta),0\leq\alpha<\beta<\infty)$ with the
covariance structure

\begin{equation}\label{eqn:strangeCorrelation}
\E\left(
\tilde\Delta(\alpha,\beta)\tilde\Delta(\alpha',\beta')
\right)
=
\left\{
\begin{array}{cl}
1&\mbox{if $\alpha=\alpha'$ \mbox{and} $\beta=\beta'$}\\
1/2&\mbox{if $\alpha=\alpha'$ \mbox{and} $\beta\neq\beta'$}\\
1/2&\mbox{if $\alpha\neq\alpha'$ \mbox{and} $\beta=\beta'$}\\
-1/2&\mbox{if $\beta=\alpha'$}\\
0&\mbox{elsewhere}
\end{array}
\right..
\end{equation}
\end{cor}

This correlation structure is surprising : for example
$\tilde\Delta(\alpha,\beta)$ and $\tilde\Delta(\alpha',\beta')$ are independent
if the segment $[\alpha,\beta]$ is strictly included in $[\alpha',\beta']$, and
positively correlated if this inclusion is not strict.
Note that there is again an effective construction of $\tilde\Delta$ : if $(\tilde D_\delta,\delta\geq 0)$
is a real valued {\it process} with all coordinates independent centered Gaussians with variance $\E(\tilde D_\delta^2)=1/2$, then
$$\tilde\Delta(\alpha,\beta)=\tilde D_\beta-\tilde D_\alpha$$
has the required correlation structure.
Concerning the discovery of this
exotic Gaussian correlation function in the context of unitary matrices, see the remark after
Theorem \ref{thm:jointRMT}.

\subsection{Analogous result on random matrices}\label{subsection:AnalogueRMT}

We note $Z(u_n,X)$ the characteristic polynomial of a matrix $u_n\in U(n)$,
and often abbreviate it as $Z$.
Theorem \ref{thm:jointNT} was inspired by the following analogue (Theorem \ref{thm:jointRMT})
in random matrix theory. This confirms the validity of the correspondence
$$
n\leftrightarrow\log t
$$
between the dimension of random matrices and the length of integration on the
critical axis, but it also supports this analogy at a local scale,
for the evaluation points of $\log Z$ and $\log \zeta$ : the  necessary shifts
are strictly analogue both for the abscissa$\setminus$radius ($\epsilon_n\setminus\epsilon_t$)
and the ordinate$\setminus$angle ($f^{(i)}\setminus\phi^{(i)}$).

\begin{thm} \label{thm:jointRMT}
Let $u_n\sim\mu_{U(n)}$, $\epsilon_n\to 0$, $\epsilon_n\gg 1/n$, and functions
$0\leq\phi^{(1)}_n< \dots< \phi^{(\ell)}_n<2\pi-\delta$ for some $\delta>0$.
Suppose that for all $i\neq j$
\begin{equation}\label{eqn:conditionRMT}
\frac{\log|\phi^{(j)}_n-\phi^{(i)}_n|}{\log \epsilon_n}\to c_{i,j}\in[0,\infty].
\end{equation}
Then the vector
\begin{equation}\label{eqn:resultRMT}
\frac{1}{\sqrt{-\log\epsilon_n}}\left(
\log Z(u_n,e^{\epsilon_n+\ii\phi^{(1)}_n}),\dots,\log Z(u_n,e^{\epsilon_n+\ii\phi^{(\ell)}_n})
\right)
\end{equation}
converges in law to a complex Gaussian vector with mean 0 and covariance function (\ref{eqn:covariance}).
Moreover, the above result remains true if $\epsilon_n\ll 1/n$,
replacing the normalization $-\log \epsilon_n$ with $\log n$ in (\ref{eqn:conditionRMT}) and (\ref{eqn:resultRMT}).
\end{thm}

\begin{rem}
Let $N_n(\alpha,\beta)$ be the number of eigenvalues $e^{\ii\theta}$ of $u_n$ with $\alpha<\theta<\beta$, and
$\delta_n(\alpha,\beta)=N_n(\alpha,\beta)-\E_{\mu_{U(n)}}(N_n(\alpha,\beta))$. Then, a little calculation (see
\cite{HKO}) yields
$$
\delta_n(\alpha,\beta)=\frac{1}{\pi}\left(\Im\log Z(u_n,e^{\ii\beta})-\Im\log Z(u_n,e^{\ii\alpha})\right)
$$
This and the above theorem imply that, as $n\to\infty$,  the vector
$$
\frac{1}{\sqrt{\log n}}\left(\delta_n(\phi^{(1)}_n,\phi^{(2)}_n),\delta_n(\phi^{(2)}_n,\phi^{(3)}_n),\dots,
\delta_n(\phi^{(\ell-1)}_n,\phi^{(\ell)}_n)\right).
$$
converges in law to a Gaussian limit.
Central limit theorems for the counting-number of eigenvalues in intervals were discovered by Wieand \cite{Wieand}
in the special case when all the intervals have a fixed length independent of $n$ (included in the case $c_{i,j}=0$ for all $i$, $j$).
Her result was extended by Diaconis and Evans to the case $\phi^{(i)}_n=\phi^{(i)}/K_n$
for some $K_n\to\infty$, $K_n/n\to 0$
(i.e. $c_{i,j}$ is a constant independent of $i$ and $j$) : Corollary \ref{cor:CountingNT} is a number-theoretic analogue of their
Theorem 6.1 in \cite{DE}.

Note that, in the general case of distinct $c_{i,i+1}$'s,
a similar result holds but the correlation function of the limit vector is not as simple as the one in
Corollary \ref{cor:CountingNT} : it strongly depends on the relative orders of these coefficients $c_{i,i+1}$'s.
\end{rem}

\subsection{Definitions, organization of the paper}\label{subsection:Definitions}

In this paper, for more concision we will make use of the following standard definition of complex Gaussian random variables.

\begin{defn}
A complex standard normal random variable $Y$ is defined as $\frac{1}{\sqrt{2}}(\mathcal{N}_1+\ii\mathcal{N}_2)$,
$\mathcal{N}_1$ and $\mathcal{N}_2$ being independent real standard normal variables. For any $\lambda,\mu\in\CC$,
we will say that $\lambda+\mu Y$ is a complex normal variable with mean $\lambda$ and variance $|\mu|^2$.
The covariance of two complex Gaussian variables $Y$ and $Y'$ is defined as
$\cov(Y,Y')=\E(\overline{Y}Y')-\E(\overline{Y})\E(Y')$, and $\Var(Y)=\cov(Y,Y)$.

A vector $(Y_1,\dots,Y_\ell)$ is a complex Gaussian vector if any linear combination of its coordinates is a
complex normal variable. For such a complex Gaussian vector and
any $\mu=(\mu_1,\dots,\mu_\ell)\in\CC_+^\ell$, $\sum_{k=1}^{\ell} \mu_k Y_k$ has variance
$\overline{\mu} C\tra{\mu}$, where $C$ is said to be the covariance matrix of $(Y_1,\dots,Y_\ell)$ : $C_{i,j}=\cov(Y_i,Y_j)$.
\end{defn}

As in the real case, the mean and the covariance matrix characterize a complex Gaussian vector.\\

Moreover, precise definitions of $\log\zeta$ and $\log Z(X)$ are necessary :
for $\sigma\geq 1/2$, we use the standard definition
$$\log \zeta(\sigma+\ii t)=-\int_{\sigma}^\infty \frac{\zeta'}{\zeta}(s+\ii t)\dd s$$
if $\zeta$ has no zero with ordinate $t$. Otherwise,
$\log \zeta(\sigma+\ii t)=\lim_{\epsilon\to 0}\log \zeta(\sigma+\ii (t+\epsilon))$.

Similarly, let $u\sim\mu_{U(n)}$ have eigenvalues $e^{\ii\theta_1},\dots,e^{\ii\theta_n}$.
For $|X|> 1$, the principal branch of the logarithm of $Z(X)=\det(\Id-X^{-1}u)$ is chosen as
$$
\log Z(X)=\sum_{k=1}^n\log\left(1-\frac{e^{\ii\theta_k}}{X}\right)=
-\sum_{j=1}^\infty\frac{1}{j}\frac{\Tr(u^j)}{X^j}.
$$
Following Diaconis and Evans \cite{DE}, if $X_n\to X$ with $|X_n|>1$ and $|X|=1$, then $\log Z(X_n)$ converges in $L^2$ to
$-\sum_{j=1}^\infty\frac{1}{j}\frac{\Tr(u^j)}{X^j}$; therefore this is our definition of $\log Z(X)$ when $|X|=1$.\\

We will successively prove Theorems \ref{thm:jointRMT} and \ref{thm:jointNT} in the next two sections. They are
independent, but we feel that the joint central limit theorem for $\zeta$ and its analogue for the random matrices
are better understood by comparing both proofs, which are similar.
In particular Proposition \ref{prop:DiacShahNT}, which is a major step towards Theorem \ref{thm:jointNT} is a strict
number-theoretic
analogue of the Diaconis-Evans theorem used in the next section to prove Theorem \ref{thm:jointRMT}.

Finally, in Section 4, we show that the same correlation structure as (\ref{eqn:covariance}) appears in the theory of
spatial branching processes.

\section{The central limit theorem for random matrices.}

\subsection{The Diaconis-Evans method.}

Diaconis and Shahshahani \cite{DiacShah} looked at the joint moments of $\Tr{u},\Tr{u^2},\dots,\Tr{u^\ell}$ for
$u\sim\mu_{U(n)}$, and showed that any of these moments coincides with the ones of $Y_1,\sqrt{2}Y_2,\dots,\sqrt{\ell}Y_\ell$
for sufficient large $n$, the $Y_k$'s being independent standard complex normal variables.
This suggests that under general assumptions, a central limit theorem can be stated for linear combinations of these traces.

Indeed, the main tool we will use for the proof of Theorem \ref{thm:jointRMT} is the following result.

\begin{thm}[Diaconis, Evans \cite{DE}]\label{thm:DiaconisEvans}
Consider an array of complex constants $\{a_{nj}\mid n\in\NN,j\in\NN\}$. Suppose there exists $\sigma^2$
such that
\begin{equation}\label{eqn:condition1}
\lim_{n\to\infty}\sum_{j=1}^\infty |a_{nj}|^2(j\wedge n)=\sigma^2.
\end{equation}
Suppose also that there exists a sequence of positive integers $\{m_n\mid n\in\NN\}$ such that $\lim_{n\to\infty}m_n/n=0$
and
\begin{equation}\label{eqn:condition2}
\lim_{n\to\infty}\sum_{j=m_n+1}^\infty |a_{nj}|^2(j\wedge n)=0.
\end{equation}
Then $\sum_{j=1}^\infty a_{nj}\Tr{u_n^j}$ converges in distribution to $\sigma Y$, where $Y$ is a complex standard normal random variable
and $u_n\sim\mu_{U(n)}$.
\end{thm}

Thanks to the above result, to prove central limit theorems for class functions, we only need to decompose them on the basis
of the traces of successive powers. This is the method employed in the next subsections, where we treat separately
the cases $\epsilon_n\gg 1/n$ and $\epsilon_n\ll 1/n$.

\subsection{Proof of Theorem \ref{thm:jointRMT} for \boldmath$\epsilon_n\gg 1/n$\unboldmath.}

From the Cram\'er-Wald device\footnote{A Borel probability measure on $\RR^\ell$ is
uniquely determined by the family of its one-dimensional projections, that is the images of $\mu$ by
$(x_1,\dots,x_\ell)\mapsto\sum_{j=1}^\ell\lambda_j x_j$, for any vector $(\lambda_j)_{1\leq j\leq \ell}\in\RR^\ell$.}
a sufficient condition to prove Theorem \ref{thm:jointRMT} is that, for any
$(\mu_1,\dots,\mu_\ell)\in\CC^\ell$,
$$
\frac{1}{\sqrt{-\log\epsilon_n}}\sum_{k=1}^\ell\mu_k\log Z(e^{\epsilon_n+\ii\phi^{(k)}_n})=
-\sum_{j=1}^\infty\frac{1}{\sqrt{-\log\epsilon_n}}\left(\sum_{k=1}^\ell\frac{\mu_k}{je^{j(\epsilon_n+\ii\phi_n^{(k)})}}\right)\Tr(u_n^j)
$$
converges in law to a complex normal variable with mean 0 and variance
\begin{equation}\label{eqn:sigmasquare}
\sigma^2=\sum_{i=1}^\ell|\mu_i|^2+\sum_{s\neq t}\overline{\mu_s}\mu_t (c_{s,t}\wedge 1).
\end{equation}
We need to check conditions (\ref{eqn:condition1}) and (\ref{eqn:condition2}) from Theorem \ref{thm:DiaconisEvans}, with
$$a_{nj}=\frac{-1}{\sqrt{-\log\epsilon_n}}\left(\sum_{k=1}^\ell\frac{\mu_k}{je^{j(\epsilon_n+\ii\phi_n^{(k)})}}\right).$$
First, to calculate the limit of
$$\sum_{j=1}^\infty |a_{nj}|^2(j\wedge n)=
\sum_{j=1}^n  j|a_{nj}|^2+
n\sum_{j=n+1}^\infty |a_{nj}|^2,$$
note that this second term tends to 0 : if $a=(\sum_{k=1}^\ell|\mu_k|)^2$, then
$$
(-\log\epsilon_n)\,n\sum_{j=n+1}^\infty |a_{nj}|^2=n \sum_{j=n+1}^\infty \left|\sum_{k=1}^\ell\frac{\mu_k}{je^{j(\epsilon_n+\ii\phi_n^{(k)})}}\right|^2
\leq a\,n\sum_{j=n+1}^\infty\frac{1}{j^2}
\leq a
$$
so $n\sum_{j=n+1}^\infty |a_{nj}|^2\to 0$. The first term can be written
$$
(-\log\epsilon_n)\sum_{j=1}^n  j|a_{nj}|^2=\sum_{j=1}^n  j\left|\sum_{k=1}^\ell\frac{\mu_k}{j e^{j(\epsilon_n+\ii\phi_n^{(k)})}}\right|^2
=\sum_{s,t}\overline{\mu_s}\mu_t\sum_{j=1}^n\frac{1}{j}
\left(\frac{e^{\ii(\phi^{(s)}_n-\phi^{(t)}_n)}}{e^{2\epsilon_n}}\right)^j.
$$
Hence the expected limit is a consequence of the following lemma.

\begin{lem}\label{lem:RMT1}
Let $\epsilon_n\gg 1/n$, $\epsilon_n\to 0$,
$(\Delta_n)$ be a strictly positive sequence, bounded by $2\pi-\delta$ for some $\delta>0$, and
$\log\Delta_n/\log\epsilon_n\to c\in[0,\infty]$. Then
$$
\frac{1}{-\log\epsilon_n}\sum_{j=1}^n
\frac{e^{\ii j\Delta_n}}{je^{2j\epsilon_n}}\underset{n\to\infty}{\longrightarrow}c\wedge 1.
$$
\end{lem}

\begin{proof}
The Taylor expansion of $\log (1-X)$ for $|X|<1$ gives
$$
\sum_{j=1}^n
\frac{e^{\ii j\Delta_n}}{je^{2j\epsilon_n}}=
-\underbrace{\log\left(1-e^{-2\epsilon_n+\ii\Delta_{n}}\right)}_{(1)}
-\underbrace{\sum_{j=n+1}^{\infty}\frac{e^{\ii j\Delta_{n}}}{je^{2j\epsilon_n}}}_{(2)}.
$$
As $\epsilon_n>d/n$ for some constant $d>0$,
$$
|(2)|\leq \sum_{j=n+1}^\infty\frac{1}{j e^{2j\epsilon_n}}
\leq \sum_{j=n+1}^\infty\frac{1}{j e^{d\frac{j}{n}}}
\underset{n\to\infty}{\longrightarrow}\int_0^\infty\frac{\dd x}{(1+x)e^{d(1+x)}},
$$
so (2), divided by $\log\epsilon_n$, tends to 0.

We now look at the main contribution, coming from (1).
If $c>1$, then $\Delta_n=\o(\epsilon_n)$, so (1) is equivalent to $\log\epsilon_n$ as $n\to\infty$.
If $0< c<1$, then $\epsilon_n=\o(\Delta_n)$ so (1) is equivalent to $\log\Delta_n$, hence to
$c\log \epsilon_n$.
If $c=1$, (1) is equivalent to $(\log \epsilon_n)\1_{\epsilon_n\geq\Delta_n}
+(\log \Delta_n)\1_{\Delta_n>\epsilon_n}$, that is to say $\log \epsilon_n$.
Finally, if $c=0$, as $(\epsilon_n)^a\ll\Delta_n<2\pi-\delta$ for all $a>0$, $(1)=\o(\log\epsilon_n)$.
\end{proof}


The condition (\ref{eqn:condition2}) in Theorem \ref{thm:DiaconisEvans} remains to be shown.
Since we have already shown that $n\sum_{j=n+1}^\infty |a_{nj}|^2\to 0$, we look for a sequence $(m_n)$ with
$m_n/n\to 0$ and
$\sum_{j=m_n+1}^n j |a_{nj}|^2\to 0$. Writing as previously $a=(\sum_{k=1}^\ell|\mu_k|)^2$, then
$$
\sum_{j=m_n+1}^n j |a_{nj}|^2\leq \frac{a}{-\log\epsilon_n}\sum_{j=m_n+1}^n\frac{1}{j}.
$$
Hence any sequence $(m_n)$ with $m_n=\o(n)$, $(\log n-\log(m_n))/\log\epsilon_n\to 0$ is convenient, for example
$m_n=\lfloor n/(-\log\epsilon_n)\rfloor$.

\subsection{Proof of Theorem \ref{thm:jointRMT} for \boldmath$\epsilon_n\ll 1/n$\unboldmath.}

We now need to check conditions (\ref{eqn:condition1}) and (\ref{eqn:condition2})
with
$$a_{nj}=\frac{-1}{\sqrt{\log n}}\left(\sum_{k=1}^\ell\frac{\mu_k}{je^{j(\epsilon_n+\ii\phi_n^{(k)})}}\right)$$
and $\sigma^2$ as in (\ref{eqn:sigmasquare}).
In the same way as the previous paragraph, $n\sum_{j=n+1}^\infty |a_{nj}|^2\to 0$, and
(\ref{eqn:condition2}) holds with $m_n=\lfloor n/\log n\rfloor$.
So the last thing to prove is
$$
\sum_{j=1}^n j |a_{nj}|^2
=\sum_{s,t}\overline{\mu}_s\mu_t\frac{1}{\log n}\sum_{j=1}^n\frac{1}{j}\left(\frac{e^{\ii(\phi^{(s)}_n-\phi^{(t)}_n)}}{e^{2\epsilon_n}}\right)^j
\underset{n\to\infty}{\longrightarrow} \sigma^2,
$$
that is to say, writing $x_n=e^{-2\epsilon_n+\ii(\phi^{(s)}_n-\phi^{(t)}_n)}$,
$$\frac{1}{\log n}\sum_{j=1}^n\frac{x_n^j}{j}\underset{n\to\infty}{\longrightarrow}  c_{s,t} \wedge 1.$$

First note that with no restriction we can suppose $\epsilon_n=0$. Indeed, if we
write $y_n=e^{\ii(\phi^{(s)}_n-\phi^{(t)}_n)}$, and $\epsilon_n\leq b/n$ for some $b>0$
(since $\epsilon_n\ll 1/n$),
$$
\left|\sum_{j=1}^n\frac{x_n^j}{j}-\sum_{j=1}^n\frac{y_n^j}{j}\right|\leq
\sum_{j=1}^n\frac{1}{j}\left|e^{-b\frac{j}{n}}-1\right|\leq b
$$
because $|e^{-x}-1|\leq x$ for $x\geq 0$. The asymptotics of $\sum_{j=1}^n\frac{y_n^j}{j}$
are given in the next lemma, which concludes the proof.

\begin{lem}\label{lem:RMT2}
Let $(\Delta_n)$ be a strictly positive sequence, bounded by $2\pi-\delta$ for some $\delta>0$, such that $-\log\Delta_n/\log n\to c\in[0,\infty]$.
Then
$$
\frac{1}{\log n}\sum_{j=1}^n\frac{e^{\ii j\Delta_n}}{j}\underset{n\to\infty}{\longrightarrow}c\wedge 1.
$$
\end{lem}

\begin{proof}
We successively treat the cases $c>0$ and $c=0$.
Suppose first that $c>0$. By comparison between the Riemann sum
and the corresponding integral,
\begin{eqnarray*}
\left|\sum_{j=1}^n\frac{e^{\ii j \Delta_n}}{j}-\int_{\Delta_n}^{(n+1)\Delta_n}\frac{e^{\ii t}}{t}\dd t\right|
&\leq&
\sum_{j=1}^n\int_{j\Delta_n}^{(j+1)\Delta_n}\left(\frac{\left|e^{\ii j\Delta_n}-e^{\ii t}\right|}{j\Delta_n}
+\left|\frac{e^{\ii t}}{j\Delta_n}-\frac{e^{\ii t}}{t}\right|\right)\dd t\\
&\leq&\sum_{j=1}^n\frac{\Delta_n}{j}+\sum_{j=1}^n\left(\frac{1}{j}-\frac{1}{j+1}\right)\\
&\leq&\Delta_n(\log n+1)+1.
\end{eqnarray*}
As $c>0$, $\Delta_n\to 0$ so $\frac{1}{\log n}\sum_{j=1}^n\frac{e^{\ii j \Delta_n}}{j}$
has the same limit as $\frac{1}{\log n}\int_{\Delta_n}^{(n+1)\Delta_n}\frac{e^{\ii t}}{t}\dd t$ as $n\to\infty$.
If $c>1$, $n\Delta_n\to 0$ so we easily get
$$
\frac{1}{\log n}\int_{\Delta_n}^{(n+1)\Delta_n}\frac{e^{\ii t}}{t}\dd t\underset{n\to\infty}{\sim}
\frac{1}{\log n}\int_{\Delta_n}^{(n+1)\Delta_n}\frac{\dd t}{t}=\frac{\log (n+1)}{\log n}\underset{n\to\infty}{\longrightarrow}1.
$$
If $0<c<1$, $n\Delta_n\to\infty$. As $\sup_{x>1}\left|\int_1^x\frac{e^{\ii t}}{t}\dd t\right|<\infty$,
$$
\frac{1}{\log n}\int_{\Delta_n}^{(n+1)\Delta_n}\frac{e^{\ii t}}{t}\dd t\underset{n\to\infty}{\sim}
\frac{1}{\log n}\int_{\Delta_n}^{1}\frac{e^{\ii t}}{t}\dd t\underset{n\to\infty}{\sim}
\frac{1}{\log n}\int_{\Delta_n}^{1}\frac{\dd t}{t}\underset{n\to\infty}{\longrightarrow} c.
$$
If $c=1$, a distinction between the cases $n\Delta_n\leq 1$, $n\Delta_n> 1$ and the above reasoning gives 1 in the limit.\\

If $c=0$, $\Delta_n$ does not necessarily converge to 0 anymore so another method is required. An elementary summation gives
$$
\sum_{j=1}^n\frac{e^{\ii j \Delta_n}}{j}=\sum_{k=1}^n\left(\frac{1}{k}-\frac{1}{k+1}\right)\sum_{j=1}^k e^{\ii j\Delta_n}
+\frac{1}{n+1}\sum_{j=1}^n e^{\ii j\Delta_n}.
$$
We will choose a sequence $(a_n)$ ($1\leq a_n\leq n$) and bound $\sum_{j=1}^k e^{\ii j\Delta_n}$ by $k$ if $k< a_n$,
by $|(e^{\ii k\Delta_n}-1)/(e^{\ii \Delta_n}-1)|\leq 2/|e^{\ii \Delta_n}-1|$ if $a_n\leq k\leq n$. This yields
$$
\sum_{j=1}^n\frac{e^{\ii j \Delta_n}}{j}\leq \sum_{k=1}^{a_n-1}\frac{1}{k+1} +\frac{2}{|e^{\ii\Delta_n}-1|}
\sum_{k=a_n}^n\left(\frac{1}{k}-\frac{1}{k+1}\right)+1
\leq \log a_n+\frac{2}{a_n|e^{\ii\Delta_n}-1|}+1.
$$
As $\Delta_n<2\pi-\delta$, there is a constant $\lambda>0$ with $|e^{\ii\Delta_n}-1|>\lambda \Delta_n$. So
the result follows if we can find a sequence ($a_n$) such that $\frac{\log a_n}{\log n}\to 0$ and $a_n\Delta_n\log n\to\infty$,
which is true for $a_n=\lfloor 2\pi/\Delta_n\rfloor$.
\end{proof}

\section{The central limit theorem for $\zeta$}

\subsection{Selberg's method.}
Suppose the Euler product of $\zeta$ holds for $1/2\leq \Re(s)\leq 1$ (this is a conjecture) : then
$\log\zeta(s)=-\sum_{p\in\mathcal{P}}\log (1-p^{-s})$ can be approximated by
$\sum_{p\in\mathcal{P}}p^{-s}$. Let $s=1/2+\epsilon_t+\ii \omega t$ with $\omega$ uniform on $(0,1)$.
As the $\log p$'s are linearly independent over $\QQ$, the terms $\{p^{-\ii \omega t}\mid p\in\mathcal{P}\}$
can be viewed as independent uniform random variables on the unit circle as $t\to\infty$, hence it was a natural thought
that a central limit theorem might hold for $\log\zeta(s)$, which was indeed shown by Selberg \cite{SelbergCLTOriginalPaper}.

The crucial point to get such arithmetical central limit theorems
is the approximation by sufficiently short Dirichlet series.
Selberg's ideas to approximate $\log\zeta$ appear in Goldston \cite{Goldston},
Joyner \cite{Joyner}, Tsang \cite{Tsang} or
Selberg's original paper \cite{SelbergCLTOriginalPaper}.
More precisely, the explicit formula for $\zeta'/\zeta$,
by Landau, gives such an approximation ($x>1$, $s$ distinct from 1, the zeros $\rho$ and $-2n$, $n\in\NN$) :
$$
\frac{\zeta'}{\zeta}(s)=-\sum_{n\leq x}\frac{\Lambda(n)}{n^s}+\frac{x^{1-s}}{1-s}-\sum_\rho \frac{x^{\rho-s}}{\rho-s}+\sum_{n=1}^\infty
\frac{x^{-2n-s}}{2n+s},
$$
from which we get an approximate formula for $\log \zeta (s)$ by integration. However, the
sum over the zeros is not absolutely convergent, hence this formula is not sufficient.
Selberg found a slight change in the above formula, that makes a great difference because all
infinite sums are now absolutely convergent : under the above hypotheses, if
$$
\Lambda_x(n)=
\left\{
\begin{array}{cc}
\Lambda(n)&\mbox{for}\ 1\leq n\leq x,\\
\Lambda(n)\frac{\log\frac{x^2}{n}}{\log n}&\mbox{for}\ x\leq n\leq x^2,
\end{array}
\right.
$$
then
\begin{multline*}
\frac{\zeta'}{\zeta}(s)=-\sum_{n\leq x^2}\frac{\Lambda_x(n)}{n^s}+\frac{x^{2(1-s)}-x^{1-s}}{(1-s)^2\log x}
+\frac{1}{\log x}\sum_\rho \frac{x^{\rho-s}-x^{2(\rho-s)}}{(\rho-s)^2}\\+
\frac{1}{\log x}\sum_{n=1}^\infty
\frac{x^{-2n-s}-x^{-2(2n+s)}}{(2n+s)^2}.
\end{multline*}
Assuming the Riemann hypothesis, the above formulas give a simple expression for $(\zeta'/\zeta)(s)$ for $\Re(s)\geq 1/2$ :
for $x\to\infty$, all terms in the infinite sums converge to 0 because $\Re(\rho-s)<0$.
By subtle arguments, Selberg showed that,
although RH is necessary for the {\it almost sure} coincidence between $\zeta'/\zeta$ and its Dirichlet series,
it is not required in order to get {\it a good $L^k$ approximation}. In particular, Selberg \cite{SelbergCLTOriginalPaper} (see also Joyner \cite{Joyner}
for similar results for more general L-functions) proved that for any $k\in\NN^*$, $0<a<1$,
there is a constant $c_{k,a}$ such that for any $1/2\leq \sigma\leq 1$, $t^{a/k}\leq x\leq t^{1/k}$,
$$
\frac{1}{t}\int_{1}^t\left|\log\zeta(\sigma+\ii s)-\sum_{p\leq x}\frac{p^{-\ii s}}{p^\sigma}\right|^{2k}\dd s\leq c_{k,a}.
$$
In the following, we only need the case $k=1$ in the above formula : with the notations of Theorem \ref{thm:jointNT}
($\omega$ uniform on $(0,1)$),
$$
\log\zeta\left(\frac{1}{2}+\epsilon_t+\ii f^{(j)}_t+\ii \omega t\right)-\sum_{p\leq t}\frac{p^{-\ii \omega t}}{p^{\frac{1}{2}+\epsilon_t+\ii f^{(j)}_t}}
$$
is bounded in $L^2$, and after normalization by $\frac{1}{-\log\epsilon_t}$ or $\frac{1}{\log\log t}$,
it converges in probability to 0. Hence, Slutsky's lemma and the Cramér-Wald device allow us to reformulate Theorem \ref{thm:jointNT}
in the following way.

\newtheorem*{Equivalent}{Equivalent of Theorem \ref{thm:jointNT}}
\begin{Equivalent}
Let $\omega$ be uniform on $(0,1)$, $\epsilon_t\to 0$, $\epsilon_t\gg 1/\log t$, and functions
$0\leq f^{(1)}_t< \dots< f^{(\ell)}_t<c<\infty$.
Suppose (\ref{eqn:conditionNT}).
Then for any finite set of complex numbers $\mu_1,\dots,\mu_\ell$,
\begin{equation}\label{eqn:equivalent}
\frac{1}{\sqrt{-\log\epsilon_t}}
\sum_{j=1}^\ell\mu_j \sum_{p\leq t}\frac{p^{-\ii \omega t}}{p^{\frac{1}{2}+\epsilon_t+\ii f^{(j)}_t}}
\end{equation}
converges in law to a complex Gaussian variable with mean 0 and variance
$$
\sigma^2=\sum_{j=1}^\ell|\mu_j|^2+\sum_{j\neq k} \overline{\mu_j}\mu_k(1\wedge c_{j,k}).
$$
If $\epsilon_t\ll 1/\log t$, then the same result holds with normalization
$1/\sqrt{\log\log t}$ instead of $1/\sqrt{-\log\epsilon_t}$ in (\ref{eqn:equivalent}) and (\ref{eqn:conditionNT}).
\end{Equivalent}

To prove this convergence in law, we need a number-theoretic analogue of Theorem \ref{thm:DiaconisEvans},
stated in the next paragraph.

\subsection{An analogue of the Diaconis-Evans theorem.}

Heuristically, the following proposition stems
from the linear independence of the $\log p$'s over $\QQ$, and the main tool to prove it is
the Montgomery-Vaughan theorem.

Note that, generally, convergence to normal variables in a number-theoretic context
is proved thanks to the convergence of all moments (see e.g. \cite{HNY}). The result below is a
tool showing that  testing the $L^2$-convergence is sufficient.

\begin{prop}\label{prop:DiacShahNT}
Let $a_{pt}$ ($p\in\mathcal{P},t\in\RR^+$) be complex numbers with $\sup_{p}|a_{pt}|\to 0$ and
$\sum_p |a_{pt}|^2\to\sigma^2$ as $t\to\infty$.
Suppose also the existence of $(m_t)$ with $\log m_t/\log t\to 0$ and
\begin{equation}\label{eqn:condition2NT}
\sum_{p>m_t} |a_{pt}|^2 \left(1+\frac{p}{t}\right)\underset{t\to\infty}{\longrightarrow}0.
\end{equation}
Then, if $\omega$ is a uniform random variable on $(0,1)$,
$$\sum_{p\in\mathcal{P}} a_{pt} p^{-\ii \omega t}\convlaw\sigma Y$$
as $t\to\infty$, $Y$ being a standard complex normal variable.
\end{prop}

\begin{rem}
The condition $m_n=\o(n)$ in Theorem \ref{thm:DiaconisEvans} is replaced here
by $\log m_t=\o(\log t)$. A systematic substitution $n\leftrightarrow \log t$
would give the stronger condition $m_t/\log m_t=\o(\log t)$ : the above proposition gives a better result
than the one expected from the analogy between random matrices and number theory.
\end{rem}

\begin{proof}
Condition  (\ref{eqn:condition2NT}) first allows to restrict the infinite sum over the set of primes $\mathcal{P}$
to the finite sum over $\mathcal{P}\cap[2,m_t]$. More precisely,
following \cite{MontgomeryVaughan}, let $(a_r)$ be complex numbers, $(\lambda_r)$ distinct real numbers and
$$
\delta_r=\min_{s\neq r}|\lambda_r-\lambda_s|.
$$
The Montgomery-Vaughan theorem  states that
$$
\frac{1}{t}\int_{0}^t\left|\sum_r a_r e^{\ii \lambda_r s}\right|^2\dd s
=
\sum_r|a_r|^2\left(1+\frac{3\pi\theta}{t\delta_r}\right)
$$
for some $\theta$ with $|\theta|\leq 1$. We substitute above $a_r$ by $a_{pt}$ and $\lambda_r$ by $\log p$,
and restrict the sum to the $p$'s greater than $m_t$ : there is a constant $c>0$ independent of $p$ with
$
\min_{p'\neq p}|\log p-\log p'|>\frac{c}{p}
$, so
$$
\frac{1}{t}\int_{0}^t\left|\sum_{p>m_t} a_{pt} p^{-\ii s}\right|^2\dd s
\leq
\sum_p|a_{pt}|^2\left(1+c'\frac{ p}{t}\right)
$$
with $c'$ bounded by $3\pi c$. Hence the hypothesis (\ref{eqn:condition2NT}) implies that $\sum_{p>m_t}a_{pt}p^{-\ii \omega t}$
converges to $0$ in $L^2$, so by Slutsky's lemma it is sufficient to show that
\begin{equation}\label{eqn:toProve}
\sum_{p\leq m_t} a_{pt}p^{-\ii \omega t}\convlaw\sigma Y.
\end{equation}
As $\sum_{p\leq m_t}|a_{pt}|^2\to\sigma^2$ and $\sup_{p\leq m_t}|a_{pt}|\to 0$,
Theorem 4.1 in Petrov \cite{Petrov} gives the following central limit theorem :
\begin{equation}\label{eqn:Proved}
\sum_{p\leq m_t} a_{pt}e^{\ii\omega_p}\convlaw\sigma Y,
\end{equation}
where the $\omega_p$'s are independent uniform random variables on $(0,2\pi)$.
The $\log p$'s being linearly independent over $\QQ$, it is well known that as $t\to\infty$
any given finite number of the $p^{\ii \omega t}$'s are asymptotically independent and uniform on the unit circle.
The problem here is that the number of these random variables increases as they become independent.
If this number increases sufficiently slowly ($\log m_t/\log t\to 0$), one can expect that
(\ref{eqn:Proved}) implies (\ref{eqn:toProve}).

The method of moments tells us that , in order to prove the central limit theorem (\ref{eqn:toProve}),
it is sufficient to show for all positive integers
$a$ and $b$ that
$$
\E\left(f_{a,b}\left(\sum_{p\leq m_t} a_{pt}p^{-\ii \omega t}\right)\right)\underset{t\to\infty}
{\longrightarrow}\E\left(f_{a,b}(\sigma Y)\right),
$$
with $f_{a,b}(x)=x^a\overline{x}^b$. From (\ref{eqn:Proved}) we know that
$$
\E\left(f_{a,b}\left(\sum_{p\leq m_t} a_{pt}e^{\ii\omega_p} \right)\right)\underset{n\to\infty}
{\longrightarrow}\E\left(f_{a,b}(\sigma Y)\right).
$$
Hence it is sufficient for us to show that, for every $a$ and $b$,
\begin{equation}
\left|\E\left(f_{a,b}\left(\sum_{p\leq m_t} a_{pt}p^{-\ii \omega t}\right)\right)-
\E\left(f_{a,b}\left( \sum_{p\leq m_t} a_{pt}e^{\ii\omega_p} \right)\right)\right|
\underset{n\to\infty}{\longrightarrow}
0.
\end{equation}

Let $n_t=|\mathcal{P}\cap [2,m_t]|$ and, for $z=(z_1,\dots,z_{n_t})\in\RR^{n_t}$,
write $f_{a,b}^{(t)}(z)=f_{a,b}\left(\sum_{p\leq m_t} a_{pt} e^{\ii z_p}\right)$, which is
$\mathscr{C}^\infty$ and $(2\pi\ZZ)^{n_t}$-periodic. Let its Fourier decomposition be
$f_{a,b}^{(t)}(z)=\sum_{k\in\ZZ^{n_t}}u_{a,b}^{(t)}(k)e^{\ii k\cdot z}$. If we write
$T^s$ for the translation on $\RR^{n_t}$ with vector $s\,p^{(t)}=s (\log p_1,\dots,\log p_{n_t})$,
inspired by the proof of Theorem 2.1 we can write
the LHS of the above equation as ($\mu^{(t)}$ is the uniform distribution on the Torus with dimension $n_t$)
$$
\left|
\frac{1}{t}\int_{0}^{t}\dd s f_{a,b}^{(t)}(T^s 0)-\int\mu^{(t)}(\dd z) f_{a,b}^{(t)}(z)
\right|
=
\frac{1}{t}\left|\sum_{k\in\ZZ^{n_t},k\neq 0}u_{a,b}^{(t)}(k)\frac{e^{\ii  t k\cdot p^{(t)}}-1}{k\cdot p^{(t)}}\right|.
$$
Our theorem will be proven if the above difference between a mean in time and a mean in space converges to 0, which
can be seen as an ergodic result.
The above RHS is clearly bounded by
$$\frac{2}{t}\left(\sum_{k\in\ZZ^{n_t}}|u_{a,b}^{(t)}(k)|\right)
\frac{1}{\inf_{k\in \mathcal{H}_{a,b}^{(t)}}
|k\cdot p^{(t)}|},$$
where $\mathcal{H}_{a,b}^{(t)}$ is
the set  of the non-zero $k$'s in $\ZZ^{n_t}$ for which
$u_{a,b}^{(t)}(k)\neq 0$ :
such a $k$ can be written $k^{(1)}-k^{(2)}$, with $k^{(1)}\in \llbracket 1,a\rrbracket^{n_t}$,
$k^{(2)}\in \llbracket 1,b\rrbracket^{n_t}$, $k^{(1)}_1+\dots+k^{(1)}_{n_t}=a$, $k^{(2)}_1+\dots+k^{(2)}_{n_t}=b$.

First note that, as
$\sum_{k\in\ZZ^{n_t}}u_{a,b}^{(t)}(k)e^{\ii k\cdot z}=\left(\sum_{p\leq m_t}a_{pt}e^{\ii z_p}\right)^a
\left(\sum_{p\leq m_t}\overline{a_{pt}}e^{-\ii z_p}\right)^b,$
$$
\sum_{k\in\ZZ^{n_t}}|u_{a,b}^{(t)}(k)|\leq \left(\sum_{p\leq m_t}|a_{pt}|\right)^{a+b}
\leq m_t^{\frac{a+b}{2}}\left(\sum_{p\leq m_t}|a_{pt}|^2\right)^{\frac{a+b}{2}}
$$
hence for sufficiently
large $t$
$$
\left|
\frac{1}{t}\int_{0}^{t}\dd s f_{a,b}^{(t)}(T^s 0)-\int\mu^{(t)}(\dd z) f_{a,b}^{(t)}(z)
\right|
\leq
\frac{2(2\sigma)^{\frac{a+b}{2}}}{t}\frac{m_t^{\frac{a+b}{2}}}{\inf_{k\in \mathcal{H}_{a,b}^{(t)}}
|k\cdot p^{(t)}|}.
$$
Lemma \ref{lem:arith} below and the condition $\log m_t/\log t\to 0$ show that the above term tends to 0, concluding the proof.
\end{proof}

\begin{lem}\label{lem:arith} For $n\geq 1$ and all $k\in\mathcal{H}_{a,b}^{t}$,
$$|k\cdot p^{(t)}|\geq \frac{1}{{n_t}^{2 \max(a,b) }}.$$
\end{lem}

\begin{proof}
For $k\in\ZZ^{n_t}$, $k\neq 0$, let $\mathcal{E}_1$ (resp $\mathcal{E}_2$) be the set of indexes $i \in\llbracket 1,n_t\rrbracket$
with $k_i$ strictly positive
(resp strictly negative.) Write $u_1=\prod_{i\in\mathcal{E}_1}p_i^{|k_i|}$
and $u_2=\prod_{i\in\mathcal{E}_2}p_i^{|k_i|}$. Suppose $u_1\geq u_2$. Thanks to the uniqueness of decomposition
as product of primes,
$u_1\geq u_2+1$.
Hence,
\begin{multline*}|k\cdot p^{(t)}|=(u_1-u_2)\frac{\log u_1-\log u_2}{u_1-u_2}\geq (\log'u_1) (u_1-u_2)\\
\geq \frac{1}{u_1}=e^{-\sum_{i\in\mathcal{E}_1}k_i\log p_i}
\geq e^{-\log p_{n_t}\sum_{i\in\mathcal{E}_1}k_i}.
\end{multline*}
For all $n_t\geq 0$, $\log p_{n_t}\leq 2\log {n_t}$. Moreover, from the decomposition $k=k^{(1)}-k^{(2)}$ in
the previous section, we know that $\sum_{i\in\mathcal{E}_1}k_i\leq a$, so
$$
|k\cdot p^{(t)}|\geq e^{-2a \log n_t}.
$$
The case $u_1< u_2$ leads to $|k\cdot p^{(t)}|\geq e^{-2 b \log n_t},$ which completes the proof.
\end{proof}

In the above proof, we showed that the remainder terms ($p>m_t$) converge to 0 in the $L^2$-norm to simplify
a problem of convergence of a sum over primes : this method seems to appear for the first time in Soundararajan \cite{SoundMoments}.

\subsection{Proof of Theorem \ref{thm:jointNT} for \boldmath$\epsilon_t\gg 1/\log t$\unboldmath.}

To prove our equivalent of Theorem \ref{thm:jointNT}, we apply the above Proposition \ref{prop:DiacShahNT} to
the random variable (\ref{eqn:equivalent}), that is to say
$$
a_{pt}=
\frac{1}{\sqrt{-\log\epsilon_t}}
\sum_{j=1}^\ell\frac{\mu_j}{p^{\frac{1}{2}+\epsilon_t+\ii f^{(j)}_t}}
$$
if $p\leq t$, 0 if $p>t$.
Then clearly $\sup_{p}|a_{pt}|\to 0$ as $t\to\infty$. For any sequence  $0<m_t<t$, writing $a=(\sum_{k=1}^\ell|\mu_k|)^2$,
$$
\sum_{m_t<p<t}|a_{pt}|^2\left(1+\frac{p}{t}\right)
\leq
\frac{a}{-\log\epsilon_t}\sum_{m_t<p<t}\frac{1}{p}+\frac{a}{-\log\epsilon_t}.
$$
As $\sum_{p\leq t}\frac{1}{p}\sim\log\log t$, condition (\ref{eqn:condition2NT}) is satisfied
if we can find $m_t=\exp(\log t/b_t)$ with $b_t\to\infty$ and $\frac{\log b_t}{-\log \epsilon_t}\to 0$ :
$b_t=-\log \epsilon_t$ for example.

We now only need to show that
$
\sum_{p\leq t}|a_{pt}|^2\to \sum_{j=1}^\ell|\mu_j|^2+\sum_{s\neq t} \overline{\mu_s}\mu_t(1\wedge c_{s,t}),
$
which is a consequence of the following lemma.

\begin{lem}\label{lem:NT1}
Let $(\Delta_t)$ be bounded and positive.
If $\epsilon_t\to 0$, $\epsilon_t\gg 1/\log t$ and $\log \Delta_t/\log \epsilon_t\to c\in[0,\infty]$, then
$$
\frac{1}{-\log\epsilon_t}\sum_{p\leq
t}\frac{p^{\ii\Delta_t}}{p^{1+2\epsilon_t}}
\underset{t\to\infty}{\longrightarrow} c\wedge 1.
$$
\end{lem}

\begin{proof}
The first step consists in showing that $\frac{1}{-\log\epsilon_t}\sum_{p\leq t}\frac{p^{\ii\Delta_t}}{p^{1+2\epsilon_t}}$
has the same limit as the infinite sum
$\frac{1}{-\log\epsilon_t}\sum_{p\in\mathcal{P}}\frac{p^{\ii\Delta_t}}{p^{1+2\epsilon_t}}$. In fact, a stronger result holds :
as $\epsilon_t$ is sufficiently large ($\epsilon_t>d/\log t$ for some $d>0$),
$\sum_{p> t}\frac{p^{\ii\Delta_t}}{p^{1+2\epsilon_t}}$ is uniformly bounded :
\begin{eqnarray*}
\sum_{p> t}\frac{1}{p^{1+2\epsilon_t}}
&=&\sum_{n> t}\frac{\pi(n)-\pi(n-1)}{n^{1+2\epsilon_t}}\\
&=&\sum_{n> t}\pi(n)\left(\frac{1}{n^{1+2\epsilon_t}}-\frac{1}{(n+1)^{1+2\epsilon_t}}\right)+\o(1)\\
&=&(1+2\epsilon_t)\int_{t}^\infty\frac{\pi(x)}{x^{2+2\epsilon_t}}\dd x+\o(1),
\end{eqnarray*}
and this last term is bounded, for sufficiently large $t$
(remember that $\pi(x)\sim x/\log x$ from the prime number theorem), by
$$2\int_t^\infty\frac{\dd x}{x^{1+\frac{d}{\log t}}\log x}
=-2\int_{0}^{e^{-d}}\frac{\dd y}{\log y}
<
\infty,$$
as shown by the change of variables $y=x^{-d/\log t}$.
Therefore the lemma is equivalent to
$$
\frac{1}{-\log\epsilon_t}\sum_{p\in
\mathcal{P}}\frac{p^{\ii\Delta_t}}{p^{1+2\epsilon_t}}
\underset{t\to\infty}{\longrightarrow} c\wedge 1.
$$
The above term
has the same limit as
$$\frac{1}{\log\epsilon_t}\sum_{p\in
\mathcal{P}}\log\left(1-\frac{p^{\ii\Delta_t}}{p^{1+2\epsilon_t}}\right)=
\frac{1}{-\log\epsilon_t}\log \zeta(1+2\epsilon_t-\ii\Delta_t)
$$
because $\log(1-x)=-x+O(|x|^2)$ as $x\to 0$, and $\sum_{p}1/p^2<\infty$.
The equivalent $\zeta(1+x)\sim 1/x$ ($x\to 0$) and the condition $\log \Delta_t/\log \epsilon_t\to c$
yield the conclusion, exactly as in the end of the proof of Lemma \ref{lem:RMT1}.
\end{proof}

\subsection{Proof of Theorem \ref{thm:jointNT} for \boldmath$\epsilon_t\ll 1/\log t$\unboldmath.}

The equivalent of Theorem \ref{thm:jointNT} now needs to be proven with
$$
a_{pt}=
\frac{1}{\sqrt{\log\log t}}
\sum_{j=1}^\ell\frac{\mu_j}{p^{\frac{1}{2}+\epsilon_t+\ii f^{(j)}_t}}
$$
if $p\leq t$, 0 if $p>t$.
Reasoning as in the previous paragraph, a suitable choice for $(m_t)$ is
$
m_t=\exp(\log t/\log\log t).
$
Therefore, the only remaining condition to check is that, for $(\Delta_t)$ bounded and
strictly positive such that $-\log\Delta_t/\log \log t\to c$ and $\epsilon_t\ll 1/\log t$,
$$
\frac{1}{\log \log t}\sum_{p\leq t}\frac{p^{\ii\Delta_t}}{p^{1+\epsilon_t}}\underset{t\to\infty}{\longrightarrow}c\wedge 1.
$$
First note that we can suppose $\epsilon_t=0$, because (using $\epsilon_t<d/\log t$ for some $d>0$ and once again $|1-e^{-x}|<x$ for $x>0$)
$$
\left|\sum_{p\leq t}\frac{p^{\ii\Delta_t}}{p^{1+\epsilon_t}}-\sum_{p\leq t}\frac{p^{\ii\Delta_t}}{p^{1}}\right||
\leq \sum_{p\leq t}\frac{\epsilon_t\log p}{p}\leq \frac{d}{\log t}\sum_{p<t}\frac{\log p}{p}\to d
$$
where the last limit makes use of the prime number theorem. The
result therefore follows from the lemma below, a strict analogue of
Lemma \ref{lem:RMT2} used in the context of random matrices.

\begin{lem}\label{lem:NT2}
Let $(\Delta_t)$ be bounded and positive, such
that $-\log\Delta_t/\log \log t\to c\in[0,\infty]$. Then
$$
\frac{1}{\log \log t}\sum_{p\leq t}\frac{p^{\ii\Delta_t}}{p}\underset{t\to\infty}{\longrightarrow}c\wedge 1.
$$
\end{lem}

\begin{proof}
As calculated in the proof of Lemma \ref{lem:NT1},
$$
\sum_{p\leq t}\frac{p^{\ii\Delta_t}}{p}=\sum_{n\leq t}\frac{n^{\ii\Delta_t}}{n}
(\pi(n)-\pi(n-1))
=(1-\ii\Delta_t)\int_{e}^t\frac{\pi(x)x^{\ii\Delta_t}}{x^2}\dd x +\o(1).
$$
The prime number theorem ($\pi(x)\sim x/\log x$) thus implies
\begin{multline*}
\sum_{p\leq t}\frac{p^{\ii\Delta_t}}{p}=
(1-\ii\Delta_t)\int_e^t\frac{x^{\ii\Delta_t}\dd x}{x\log x}+(1-\ii\Delta_t)\,\o\left(\int_e^t\frac{\dd x}{x\log x}\right)+\o(1)\\
=(1-\ii\Delta_t)\int_{\Delta_t}^{\Delta_t\log t}\frac{e^{\ii y}\dd y}{y}+(1-\ii\Delta_t)\,\o(\log\log t)+\o(1).
\end{multline*}
If $c>1$, $\Delta_t\log t\to 0$, so the above term is equivalent to $\int_{\Delta_t}^{\Delta_t\log t}\dd y/y=\log\log t$. If $c<1$,
$\Delta_t\log t\to\infty$ so, as $\sup_{x>1}\left|\int_1^x\frac{e^{\ii y}}{y}\dd y\right|<\infty$,
$\frac{1}{\log \log t}\sum_{p\leq t}\frac{p^{\ii\Delta_t}}{p}$ tends to the same limit as
$\int_{\Delta_t}^{1}\dd y/y=\log \Delta_t/\log\log t\to c$.
Finally, if $c=1$, the distinction between the cases $\Delta_t\log t>1$ and $\Delta_t\log t<1$ and the above reasoning give 1 in the limit.
\end{proof}

\section{Connection with spatial branching
processes.}\label{Section:SpatialBranchingProcesses}

There is no easy a priori reason why the matrix
(\ref{eqn:covariance}) is a covariance matrix.
More precisely, given positive numbers $c_1,\dots,c_{\ell-1}$,
is there a reason why the symmetric matrix
$$
C_{i,j}=\E(\overline{Y_i}Y_j)=\left\{
\begin{array}{ccc}
1&\mbox{if}&i=j\\
1\wedge \inf_{\llbracket i,j-1\rrbracket}c_k&\mbox{if}&i<j
\end{array}
\right.
$$
is positive
semi-definite ?
This is a by-product of Theorem \ref{thm:jointNT},
and a possible construction for the Gaussian vector
$(Y_1,\dots,Y_\ell)$ is as follows. Define the angles $\phi_n^{(k)}$,
$1\leq k\leq\ell$, by $\phi_n^{(1)}=0$ and
\begin{equation}\label{eqn:covar}
\phi_n^{(k)}=\phi_n^{(k-1)}+\frac{1}{n^{c_{k-1,k}}},\ 2\leq k\leq
\ell.
\end{equation}
Let $(\mathcal{X}_r)_{r\geq 1}$ be independent standard complex
Gaussian variables. For $1\leq k\leq\ell$, let
$$
Y^{(n)}_k=\frac{1}{\sqrt{\log n}}\sum_{r=1}^n e^{\ii r
\phi_n^{(k)}}\frac{\mathcal{X}_r}{\sqrt{r}}.
$$
Then $(Y^{(n)}_1,\dots,Y^{(n)}_\ell)$ is a complex Gaussian vector,
and Lemma \ref{lem:RMT2} implies that its covariance matrix
converges to (\ref{eqn:covar}).\\

Instead of finding a Gaussian vector with covariance structure (\ref{eqn:covar}),
we consider this problem :
given $c_1,\dots,c_\ell$ positive real numbers,
can we find a centered (real or complex) Gaussian vector $(X_1,\dots, X_\ell)$ with
\begin{equation}\label{eqn:discrete}
\E(X_iX_j)=\inf_{i\leq k\leq j}c_k
\end{equation}
for all $i\leq j$ ?
A matrix $C$ of type (\ref{eqn:covar}) can always be obtained as a $\lambda C'+D$ with $\lambda>0$,
$C'$ of type (\ref{eqn:discrete}) and $D$ diagonal with positive entries, so the above problem is more general than
the original one.

 Equation (\ref{eqn:discrete}) is the discrete analogue of the following problem,
considered in the context of spatial branching processes by Le Gall (see e.g. \cite{leGall}).
Strictly following his work, we note $e:[0,\sigma]\to \RR^+$ a continuous function such that $e(0)=e(\sigma)=0$. Le Gall associates to
such a function $e$ a continuous tree by the following construction :
each $s\in[0,\sigma]$ corresponds to a vertex of the tree after identification of $s$ and $t$ ($s\sim t$) if
$$e(s)=e(t)=\inf_{[s,t]} e(r).$$
This set $[0,\sigma]\slash \sim$ of vertices is endowed with the partial order $s\prec t$ ($s$ is an ancestor of $t$)
if
$$e(s)=\inf_{[s,t]}e(r).$$
Independent Brownian motions can diffuse on the distinct branches of the tree : this
defines a Gaussian process $B_u$ with $u\in[0,\sigma]\slash \sim$ (see \cite{leGall} for the construction of this diffusion).
For $s\in[0,\sigma]$ writing $X_s=B_{\overline{s}}$ (where $\overline{s}$ is the equivalence class of $s$ for $\sim$), we get a
continuous centered Gaussian process on $[0,\sigma]$ with correlation structure
\begin{equation}\label{eqn:continuous}
\E(\overline{X_s} X_t)=\inf_{[s,t]}e(u),
\end{equation}
which is the continuous analogue of (\ref{eqn:discrete}).
This construction by Le Gall yields a solution of our discrete problem (\ref{eqn:discrete}).
More precisely,
suppose for simplicity that all the $c_i$'s are distinct (this is not a restrictive hypothesis by a continuity argument),
and consider the graph $i\mapsto c_i$. We say that $i$ is an ancestor of $j$ if
$$
c_i=\inf_{k\in \llbracket i, j\rrbracket}c_k.
$$

\begin{wrapfigure}[15]{r}{.5\linewidth}
\psset{unit=0.7cm, linewidth=1pt}
\begin{center}
\begin{pspicture}(-1,-1)(8,7)

\psline{->}(0,0)(7,0)
\psline{->}(0,0)(0,7)

\psline{-*}(5,0)(5,1)
\psline{-*}(1,1)(1,2)
\psline{-*}(3,2)(3,4)
\psline{-*}(2,4)(2,5)
\psline{-*}(4,4)(4,7)
\psline{-*}(7,1)(7,3)
\psline{-*}(6,3)(6,6)

\psline[linestyle=dashed](1,1)(7,1)
\psline[linestyle=dashed](1,2)(3,2)
\psline[linestyle=dashed](6,3)(7,3)
\psline[linestyle=dashed](2,4)(4,4)

\rput(7.5,0){$k$}
\rput(0,7.5){$c_k$}

\rput(0.5,1.5){$\mathcal{N}_1$}
\rput(1.5,4.5){$\mathcal{N}_2$}
\rput(2.5,3){$\mathcal{N}_3$}
\rput(3.5,5.5){$\mathcal{N}_4$}
\rput(4.5,0.5){$\mathcal{N}_5$}
\rput(5.5,4.5){$\mathcal{N}_6$}
\rput(6.5,2){$\mathcal{N}_7$}


\end{pspicture}
\end{center}
\end{wrapfigure}

The father of $i$ is its nearest ancestor, for the distance $d(i,j)=|c_i-c_j|$. It is noted $p(i)$. We can write
$c_{\sigma(1)}<\dots<c_{\sigma(\ell)}$ for some permutation $\sigma$, and $(\mathcal{N}_1,\dots,\mathcal{N}_\ell)$ a vector of independent
centered complex Gaussian variables, $\mathcal{N}_k$ with variance $c_k-c_{p(k)}$ (by convention $c_{p(\sigma(1))}=0$).
Then the Gaussian vector $(X_1,\dots,X_\ell)$ iteratively defined by
$$
\left\{
\begin{array}{ccl}
X_{\sigma(1)}&=&\mathcal{N}_{\sigma(1)}\\
X_{\sigma(i+1)}&=&X_{p(\sigma(i+1))}+\mathcal{N}_{\sigma(i+1)}
\end{array}
\right.
$$
satisfies (\ref{eqn:discrete}), by construction.

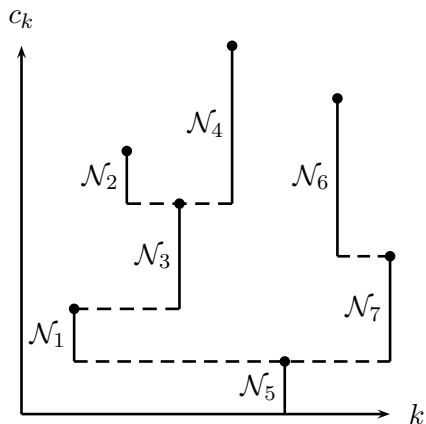
\renewcommand{\refname}{References}


\begin{thebibliography}{99}


\bibitem{CD} M. Coram, P. Diaconis, New tests of the correspondence between unitary eigenvalues and the zeros of Riemann's zeta function,
J. Phys. A: Math. Gen. 36 2883-2906.

\bibitem{DE} P. Diaconis and S. Evans, Linear Functionals of Eigenvalues of Random
Matrices, Trans of the AMS, Vol. 353, No. 7 (Jul., 2001), pp.
2615-2633

\bibitem{DiacShah} P. Diaconis, M. Shahshahani, On the eigenvalues
of random matrices, Studies in applied probability,  J. Appl.
Probab.  31A  (1994), 49-62.

\bibitem{Fujii} Hejhal D, Friedman J, Gutzwiller M and Odlyzko A (ed) 1999
Emerging Applications of Number Theory (Berlin: Springer)

\bibitem{leGall} J.F. le gall, Spatial Branching Processes, Random Snakes and Partial Differential Equations, Lectures in Mathematics, ETH
Zürich, Birkhäuser, 1999.

\bibitem{Goldston} D. A. Goldston, Notes on pair correlation of zeros and prime numbers, in Recent Perspectives in
Random Matrix Theory and Number Theory, London Mathematical Society Lecture,
Note Series 322 (CUP), (2005), edited by F. Mezzadri and  N.C. Snaith.

\bibitem{HKO} C. P. Hughes, J. P. Keating and N. O'Connell, On the characteristic polynomial of a random unitary
matrix, Comm. Math. Phys. 220, n°2, p 429-451, 2001.

\bibitem{HNY} C. P. Hughes, A. Nikeghbali, M. Yor, Probability Theory and Related Fields 141 (2008) 47-59.

\bibitem{Joyner} W. D. Joyner, Distribution theorems of L-functions , Pitman Research Notes, vol 142, 1986.

\bibitem{MontgomeryVaughan} H. L. Montgomery, R. C. Vaughan, Hilbert's inequality, J. London Math. Soc. (2), 8 (1974), 73-82.

\bibitem{Petrov}  V.V. Petrov, Limit Theorems of Probability Theory, Oxford University Press, Oxford, 1995.

\bibitem{SelbergCLTOriginalPaper} A. Selberg, Contributions to the theory of the Riemann zeta-function,
Arkiv for Mathematik og Naturvidenskab B. 48 (1946), 5, 89-155.

\bibitem{SoundMoments} K. Soundararajan, Moments of the Riemann zeta-function, to appear in Ann. Math.

\bibitem{Tit}  E.C. Titschmarsh, The Theory of the Riemann Zeta Function, London, Oxford
Unversity Press,  1951.

\bibitem{Tsang} Kai-man Tsang, The distribution of the values of
the zeta function, Thesis, Princeton University, October 1984, 179
pp.

\bibitem{Wieand} K.L. Wieand, Eigenvalue distributions of random matrices in the permutation group
and compact Lie groups, Ph.D. thesis, Harvard University, 1998.

\end{thebibliography}
\end{document}